\documentclass{amsart}

\usepackage{amssymb}
\usepackage{amsbsy}
\usepackage{amscd}
\usepackage{amsmath}
\usepackage{amsthm}
\usepackage{bbm}
\usepackage{verbatim}
\usepackage{xypic}
\usepackage{graphics}
\usepackage{verbatim}

\newtheorem{theorem}{Theorem}
\numberwithin{defn}{section}
\newtheorem{prop}{Proposition}\numberwithin{defn}{section}
\newtheorem{lemma}{Lemma}\numberwithin{defn}{section}
\newtheorem{conj}{Conjecture}\numberwithin{defn}{section}
\newtheorem{quest}{Question}\numberwithin{defn}{section}
\newtheorem{subquest}{Subquestion}\numberwithin{subquest}{quest}
\newtheorem{prob}{Problem}\numberwithin{defn}{section}

\title[Galois actions on complex braid groups]{Galois actions on complex braid groups}
\dedicatory{Dedicated to Fran\c cois Digne and Jean Michel on the occasion of their sixtieth birthday.}

\author[Ivan Marin]{Ivan Marin}

\address{Institut de Math\'ematiques de Jussieu, Universit\'e Paris 7, 175 rue du Chevaleret, F-75013 Paris}

\email{marin@math.jussieu.fr}



\def\C{\mathbbm{C}}
\def\R{\mathbbm{R}}
\def\Z{\mathbbm{Z}}
\def\Q{\mathbbm{Q}}

\def\A{\mathbbm{A}}
\def\Qb{\overline{\mathbbm{Q}}}
\def\P{\mathbbm{P}}
\def\kk{\mathbf{k}}

\def\la{\lambda}
\def\GL{\mathrm{GL}}

\def\Out{\mathrm{Out}}
\def\Aut{\mathrm{Aut}}
\def\Hom{\mathrm{Hom}}

\def\hpi{\widehat{\pi}_1}
\def\GTH{\ensuremath{\widehat{GT}}}

\def\Ker{\mathrm{Ker}}
\def\Gal{\mathrm{Gal}}
\def\Spec{\mathrm{Spec}}

\def\Ad{\mathrm{Ad}}
\def\into{\hookrightarrow}
\def\onto{\twoheadrightarrow}
\def\ss{s}

\def\GGam{ \mathrm{I} \! \Gamma}
\def\eps{\varepsilon}
\def\ii{\mathrm{i}}

\begin{document}

\begin{abstract}
We establish the faithfulness of a geometric action of the absolute Galois group
of the rationals that can be defined on the discriminantal variety associated to
a finite complex reflection group, and review some possible connections
with the profinite Grothendieck-Teichm\"uller group.

\end{abstract}

\maketitle

\section{Introduction}

More than a decade ago, Brou\'e, Malle and Rouquier suggested in \cite{BMR} that lots of the well-known properties of the braid
groups associated to real reflection groups (so-called Artin groups, or Artin-Tits groups, or Artin-Brieskorn groups),
which are consequences either of the simplicial structure of the corresponding hyperplane arrangement or of Coxeter theory,
could actually be generalized to all complex reflection groups, for which no such comprehensive theory apply.
Although a general reason is still missing for most of these properties, and although most proofs rely on a case-by-case
approach using the Shephard-Todd classification, this intuition has been confirmed by large in the past decade.
It should be noticed that the group-theoretic study of these general braid groups is still at the beginning ; for instance, the
determination of nice \emph{presentations} for all these groups has been completed only recently.

At much about the same time, Matsumoto investigated in \cite{MATSUARTIN} the geometric Galois actions on 
the discriminantal variety of Weyl groups and compared it with the ones on mapping class groups. The reason
why this article restricted itself to Weyl groups was probably twofold : the most visible connections
with mapping class groups involved Artin groups of type $A D E$, and the Weyl groups are
the reflection groups which are naturally `defined over $\Q$'.

In a joint work with J. Michel (see \cite{MARINMICHEL}) we proved that the discriminantal varieties arising from complex reflection groups
are actually defined over $\Q$. In the first part of this note we review this result, define the corresponding
Galois action, and prove its faithfulness. We also survey the state-of-the-art for the basic group-theoretic questions which are relevant to the geometric-algebraic setting. The second part of the paper is more speculative. 
It investigates these `complex braid groups', and already usual Artin groups, from a `Grothendieck-Teichm\"uller perspective' : can we expect that the Galois actions on these groups factor through $\GTH$ ?

{\bf Acknowledgements.} I thank B. Collas, C. Cornut, P. Lochak, J. Tong, and A. Tamagawa for discussions,
as well as the referee for a careful examination of the manuscript.

\section{Complex reflection groups}

Let $\kk$ be field of characteristic 0 and $V$
a finite-dimensional $\kk$-vector space.
A pseudo-reflection in $V$ is an element $s$ of $\GL(V)$ whose
set of fixed points $\Ker(s - 1)$ is an hyperplane.
A (finite) reflection group $W$ over $\kk$ is a finite subgroup of $\GL(V)$
generated by pseudo-reflections. We refer to \cite{LEHRTAY} for basic
notions in this area, and recall that these groups have been
classified by Shephard and Todd. When $\kk$ is algebraically closed,
such groups can naturally be decomposed as a product of
\emph{irreducible} ones (which act irreducibly on their
associated vector space) which belong either
to an infinite family $G(de,e,n)$ depending on three integer
parameters $d,e,n$, or to a finite set of 34 exceptions denoted $G_4$, \dots, $G_{37}$. We let $\mathcal{R}$
denote the set of pseudo-reflections of $W$,  $\mu_n(\kk)$ the group of $n$-th
roots of $1$ in $\kk^{\times}$, and $\mu_n = \mu_n(\C)$. Of course we can assume $\kk \subset \C$.

When $| \mu_{de}(\kk)| = de$, that is $\kk \supset \Q(\mu_{de})$,
the groups $G(de,e,n)<\GL_n(\kk)$ are the groups of $n \times n$
\emph{monomial} matrices in $\kk$ (one nonzero
entry per row and per column) whose
nonzero entries belong to $\mu_{de}(\kk)$ and have product
in $\mu_d(\kk)$.

When $\kk = \C$, we let $\mathcal{A} = \{ \Ker(s-1) \ | \ s \in \mathcal{R} \}$
denote the hyperplane arrangement associated to $\mathcal{R}$,
and $X = V \setminus \bigcup \mathcal{A}$ the
corresponding hyperplane complement. By a theorem
of Steinberg, the action of $W$ on $X$ is free and defines
a Galois (\'etale) covering $X \to X/W$.
Let $\alpha_H, H \in \mathcal{A}$ denote a collection
of linear forms such that $H = \Ker \alpha_H$, and
$e_H$ the order of the (cyclic) subgroup of $W$ fixing $H$. By
a theorem of Chevalley-Shephard-Todd, $V // W$ is isomorphic
over $\C$ to the affine space : identifying $V$ with $\C^n$,
there exists $y_1,\dots,y_n \in S(V^*)$
such that $S(V^*)^W = \C[y_1,\dots,y_n]$.
Letting $\Delta = \prod_{H \in \mathcal{A}} \alpha_H^{e_H} \in S(V^*)^W$,
$X/W$ is then identified to the complement of $\Delta = 0$
in $\C^n$.
The groups $B = \pi_1(X/W)$ and $P = \pi_1(X)$ are called
the braid group and pure braid group associated to $W$. As in
the case of the usual braid group (corresponding to
$W = \mathfrak{S}_n \subset \GL_n(\C)$) we have a short sequence
$1 \to P \to B \to W \to 1$. When $W$ is a Coxeter
group, that is a reflection group for $\kk = \R$, the groups
$B$ were basically introduced by Tits and extensively studied since
then.

\subsection{Arithmetic actions}

There is a natural number field $K$ associated to a complex reflection group $W < \GL_N(\C)$, called
the \emph{field of definition} of $W$ : it is the field generated by the traces of the elements of $W$. It
is a classical fact (see e.g. \cite{LEHRTAY}, theorem 1.39) that a suitable conjugation in $\GL_N(\C)$ maps
$W$ into $\GL_N(K)$. Since $W$ is finite, $K \subset \Q(\mu_{\infty})$ is an abelian extension of $\Q$. As
a consequence, $X$ and $X/W$ can be defined over $K$. One has $K \subset \R$ if and only if $W$ is a
Coxeter group, and $K = \Q$ if and only if $W$ is the Weyl group of some root system.
  
We proved in \cite{MARINMICHEL} the following.

\begin{theorem} \label{theorat} (see \cite{MARINMICHEL})The varieties $X/W$, $X$, and
the \'etale covering map $X \to X/W$ can be defined
over $\Q$.
\end{theorem}

More precisely, we proved that we can attach to each (complex) reflection group
a suitable extension $K''$ of $K$ (with $K = K''$ for most of the groups), which is
still an abelian extension of $\Q$, and which fulfills the following properties :
\begin{itemize}
\item Up to global conjugation by some element of
$\GL_n(\C)$ we have and can assume $W < \GL_n(K'')$ with $\Gal(K''|\Q).W = W$.
This defines a natural action by automorphisms of $\Gal(K''|\Q)$ on $W$,
and we can assume $V$ has a natural $K''$-form $V_0$, namely $V = V_0 \otimes_{K''} \C$ with $W < \GL(V_0)$
-- in particular, the $\alpha_H$ can be chosen in $V_0^*$.
\item This action induces injective morphisms $\Gal(K''|\Q) \into \Aut(W)$ and $\Gal(K|\Q) \into \Out(W)$.
\item The action of $\Gal(K''|\Q)$ normalizes the action of $W$ on $S(V_0^*)$,
hence $S(V_0^*)$, $S(V_0^*)^W$ have natural $\Q$-forms $S(V_0^*) = S(V_0^*)^{\Gal(K''|\Q)}\otimes K'' $, $S(V_0^*)^W = S(V_0^*)^{W,\Gal(K''|\Q)} \otimes K''$
defining $V$, $V//W = \Spec  S(V_0^*)^{W,\Gal(K''|\Q)}$ over $\Q$, as well as $X/W$.
\end{itemize}

As a consequence we get natural compatible maps
$G(\Q) \to \Out^*(\hpi (X/W))$ and $G(\Q) \to \Out^*(\hpi (X))$, where
we denote, for an algebraic variety $Y$, $\Aut^* \hpi Y$ the group of inertia-preserving automorphisms
of $\hpi Y$, and $\Out^* \hpi Y$ its image in $\Out \, \hpi Y$
(and we use the notation  $G(\Q) = \Gal(\Qb | \Q)$ for the absolute Galois group of the rationals).
Since the \'etale covering map $X \to X/W$ is defined over $\Q$, one can check on the
fiber above a rational base point that the action of $G(K'') = \Gal(\Qb | K'')$
is trivial on the covering $X \to X/W$. This induces an action by automorphisms of $\Gal(K''|\Q)$
on $W$ which is easily seen to be identical to the one defined above, and thus provides
a geometrical origin to the `Galois automorphisms' of \cite{MARINMICHEL}.
The situation is depicted in the following commutative diagrams.

$$
\xymatrix{
G(\Q) \ar[r] \ar[d] & \Aut^*(\widehat{B}) \ar[d] \\
\Gal(K''|\Q) \ar[r]& \Aut(W) 
}
\ \ 
\xymatrix{
G(\Q) \ar[r] \ar[d] & \Out^*(\widehat{B}) \ar[d] \\
\Gal(K|\Q) \ar[r]& \Out(W) 
}
$$

It turns out that these algebraic varieties, beside being rational,
are `anabelian' enough in the sense that the action of the Galois group
on their fundamental group is faithful. Informally, this `lifts'
at the level of the braid group the embedding $\Gal(K|\Q) \into \Out(W)$
defined in \cite{MARINMICHEL}.

\begin{theorem}\label{fidel} If $W$ is not abelian, then the maps $G(\Q) \to \Out(\hpi (X))$
and $G(\Q) \to \Out(\hpi (X/W))$ are injective.
\end{theorem}
\begin{proof}

We can choose a $K$-form $V = V_K \otimes_K \C$ with $W < \GL(V_K)$. 

Notice that, since $K \subset \Q(\mu_{\infty})$ then $K$ is an abelian (Galois) extension 
of $\Q$, and is stable under the complex conjugation $z \mapsto \bar{z}$. 
We can thus endow $V_K$
with a unitary
(positive definite hermitian) form which is left invariant under the finite group $W$, and extend it to $V$. For $H \subset V$,
we denote $H^{\perp} \subset V$ the orthogonal of $H$ with respect to this unitary
form.

Since $W$ is not abelian and generated by $\mathcal{R}$, there
exists $s_1,s_2 \in \mathcal{R}$ with $s_1 s_2 \neq s_2 s_1$.
Let $H_i = \Ker(s_i - 1)$ and $H = H_1 \cap H_2$. 
One can assume that $s_i$ generates the (cyclic) group of $W$ which fixes $H_i$ and more
precisely, letting $d_i$ denote the order of that group, that the non-trivial eigenvalue of $s_i$
is $\exp(2 \ii \pi/d_i)$. Such a pseudo-reflection is classically called `distinguished', and clearly such
pseudo-reflections are in 1-1 correspondence (under $s \mapsto \Ker(s-1)$) with the collection of reflecting hyperplanes.
Choosing $v_1 \in H^{\perp} \cap H_2$,
$v_2 \in H^{\perp} \cap H_1$, and $v_3,\dots,v_n$ a basis for $H$,
in such a way that $v_1,\dots,v_n \in V_{K}$ (this is possible because the hyperplanes are defined over $K$)
one gets a basis $v_1,\dots,v_n$ of $V_{K}$ and $V$, hence an identification
of $V_{K}$ and $V$ with $K^n$ and $\C^n$ respectively, and of $H_i$ with the hyperplane defined by $z_i = 0$
if $\underline{z} = (z_1,\dots,z_n)$ denotes a generic point of $K^n$.
Let $W_0$ be the subgroup of $W$ generated by $s_1,s_2$.
It is a reflection subgroup of $W$ that fixes $H$. Since $W_0$
is not abelian, it contains a pseudo-reflection with respect to
another hyperplane.
Indeed, $s_1 s_2 s_1^{-1} \in W_0$ is a distinguished reflection with reflecting hyperplane
$s_1(H_2)$, and $s_1(H_2)
 \in \{ H_1,H_2 \}$ would imply $s_1 s_2 s_1^{-1} = s_1$ or
 $s_1 s_2 s_1^{-1} = s_2$, meaning $s_2 = s_1$ or $s_1 s_2 = s_2 s_1$,
 which have been ruled out.

The equation of this new hyperplane has the form $z_1 = \alpha z_2$
with $\alpha \neq 0$ and $\alpha \in K \subset \overline{\Q}$.
We consider the morphism $V \setminus \{ \underline{z} \ | \ (z_1,z_2) = (0,0) \}
\to \P^1$ defined by $\underline{z} \mapsto [z_1 : z_2]$.
It maps $\bigcup \mathcal{A}\ $Êto a finite set $S$ of points, including
$0 = [0:1], \infty= [1:0], \alpha = [\alpha : 1]$, hence
induces 
an algebraic morphism $f : X \to \P^1 \setminus S$
defined over $K$. 
Moreover, the
induced map $f_* : \hpi X  \to \hpi \P^1 \setminus S$ is equivariant
under $G(K)$ and \emph{surjective}.

The easiest way to see this is
to prove it topologically
on the usual $\pi_1$'s
and then
use the right-exactness of the profinite completion functor. 
We do this now. Let $\mathcal{A}' = \{ H' \in \mathcal{A} \ | \ H' \not\supset H \}$, and choose
$\underline{\zeta}_0 = (0,0,\zeta_3,\dots,\zeta_n) \in H \setminus \bigcup \mathcal{A}'$. Since
$H \setminus \bigcup \mathcal{A}'$ is open in $H$, for $\eps > 0$ small
enough we have $\underline{\zeta} = (\zeta_1,\zeta_2,\dots,\zeta_n) \in X$ whenever
$0 < |\zeta_1| < \eps$, $0 < |\zeta_2| < \eps$, and $[z_1:z_2] \not\in S$. We choose such
an $\eps$. Now $\mathbbm{P}^1 \setminus S$ can be identified with $\C \setminus S_0$ under
$[z_1 : z_2] \mapsto z_1/z_2$, where $S_0$ is some finite set. Let $M = 1 + \max\{ | z | ; z \in S_0 \}$
and choose $z_0 \in \C \setminus S_0$ with $|z_0| \leq M$ for base-point. A loop $\gamma$
in $\C \setminus S_0$ based at $z_0$ can be homotoped to a loop satisfying $|\gamma(t)| \leq M +1$
for all $t$, and such a loop has a lift of the form $t \mapsto (\eps \frac{z(t)}{M},\frac{\eps}{M},\zeta_3,\dots,\zeta_n)$
in $\pi_1(X(\C),\underline{\zeta}_0(\eps))$ with $\underline{\zeta}_0(\eps) = (\eps \frac{z_0}{M}, \frac{\eps}{M},\zeta_3,\dots,
\zeta_n)$, which proves the surjectivity.

Another way
is to check that $f$ is faithfully flat and quasicompact, hence universally
submersive, and that it has geometrically connected fibers ; the conclusion then
follows from \cite{SGA1} (exp. IX cor. 3.4). Indeed, $f$ is affine, more precisely $f = \Spec \varphi$ with 
$$
\varphi : A = K[x,y,x^{-1},y^{-1}, (y- \alpha_i x)^{-1} ] \into K [z_1,\dots,z_n,\alpha_{L_1}^{-1},\dots,
\alpha_{L_m}^{-1}]
$$
defined by $x \mapsto z_1, y \mapsto z_2$, with 
$\alpha_{H_1},\dots,\alpha_{H_m}$ 
some linear forms
defining (over $K$) the hyperplanes. Identifying $A$ with its image under $\varphi$, we have
$K[z_1,\dots,z_n,\alpha_{H_1}^{-1},\dots,
\alpha_{H_m}^{-1}] = A[z_3,\dots,z_n,L_1^{-1},\dots,L_r^{-1}] = B$ with each $L_i \in A[z_3,\dots,z_n]$
of degree $1$. Now $B$ is faithfully flat over $A$ as the localization of a polynomial algebra $C = A[z_3,\dots,z_n]$
over $A$ at an element $P = L_1 \dots L_r$ with $P \neq 0$ inside $C/ \mathfrak{P} C$ for
each prime ideal $\mathfrak{P}$ of $A$. In particular $f$ is surjective ; since it is
affine it is quasi-compact if and only if
$f^{-1}(\Spec A) = \Spec B$ is quasi-compact, and this holds because $B$ is noetherian.
Finally, since $f$ is the composition of an affine open immersion
and of the projection $\mathbbm{A}^n \to \mathbbm{A}_2$, its fibers are open
subsets of an affine space, and are thus connected.

Since $|S| \geq 3$, $\P^1 \setminus S$
is an hyperbolic curve over $K$, so by Belyi's theorem and \cite{MATSUMOTO},
the natural map $G(K) \to \Out \hpi \P^1 \setminus S$ is injective,
and so is $G(K) \to \Out \hpi X = \Out \widehat{P}$ by the surjectivity
of $f_* : \hpi X \to \hpi \P^1 \setminus S$. Now $X$ is the complement
of an hypersurface in $\mathbbm{A}^n$ defined by a reduced equation of
the form $\delta=0$  for some $\delta \in \Q[x_1,\dots,x_n]$
the product of convenient $\alpha_H$, $H \in \mathcal{A}$. The function
$\delta$ induces a morphism $g : X \to \mathbbm{A}^1 \setminus \{ 0 \} = \mathbbm{P}^1
\setminus \{ 0, \infty \}$ defined over $\Q$. 
Let $\psi : G(\Q) \to \Out \hpi X = \Out \widehat{P}$. We prove that the induced
morphism $\hpi X \to \hpi \mathbbm{A}^1 \setminus \{ 0 \}$
is surjective. It is enough to prove that $f_* : \pi_1(X(\C), \underline{\zeta}) \to
\pi_1(\C \setminus \{ 0 \}, f(\underline{\zeta}))$ is surjective. For this, we
can assume $V = \C^n$, $\delta = u_0u_1\dots u_m = u_0 \hat{\delta}$
with $u_i \in V^*$ and $i \neq j \Rightarrow \Ker u_i \neq \Ker u_j$, and that
$u_0 : \underline{z} = (z_1,\dots,z_n) \mapsto z_1$
is the first coordinate. Since $\Ker u_i \neq \Ker u_0$ for all $i \neq 0$,
there exists an element $\underline{\zeta}_0 \in \Ker u_0 \setminus \bigcup_{i \neq 0} \Ker u_i$.
We let $\underline{\zeta}_{\eps} = (\eps,\zeta_2,\dots,\zeta_n)$ and
$\gamma(t) = (\eps e^{2 \ii \pi t}, \zeta_2,\dots,\zeta_n)$. We have $\delta(\gamma(t)) = 
\eps e^{2 \ii \pi t} \hat{\delta}(\gamma(t))$. For $\eps > 0$ small enough, $\hat{\delta}(\gamma)$ is close to $\hat{\delta}(\underline{\zeta}) \neq 0$,
hence $t \mapsto \delta(\gamma(t))$ describes a positive turn around $0$, whose class
generates $\pi_1(\C \setminus \{ 0 \}, \delta(\underline{\zeta}_{\eps}))$,
and is the image of $[\gamma] \in \pi_1(X(\C), \underline{\zeta}_{\eps})$.

Since the kernel
of the cyclotomic character
$G(\Q) \to \Out \hpi \P^1 \setminus \{ 0 ,\infty \}$ is $\Q(\mu_{\infty})$,
we get that every element in $\Ker \psi$ fixes $\Q(\mu_{\infty})$,
hence belongs to $G(K)$. Since we know the action of $G(K)$
is faithful this proves $G(\Q) \into \Out \hpi X = \Out \widehat{P}$.

We now consider the map $\varphi : G(\Q) \to \Out \widehat{B} = \Out \hpi X/W$,
and a lift $\tilde{\varphi} : G(\Q) \to \Aut \widehat{B}$ given by
the choice of a rational base point. We can assume that this base
point lies in the image of $X(\Q) \to (X/W)(\Q)$. 
An element 
$\sigma$ in $\Ker \varphi$ should induce on $\widehat{P}$
the conjugation $\Ad(g)$ by an element $g$ of $\widehat{B}$.
Let $\bar{g}$ its image in $\widehat{B}/\widehat{P} = W$. If $\sigma \in G(K)$,
$\Ad(\bar{g})$ is the identity automorphism, hence $\bar{g} \in Z(W)$. Now
the morphism $Z(B) \to Z(W)$ is known to be onto (see \cite{BMR}) hence
we can assume that $\sigma \mapsto \Ad(g)$ with $\bar{g} = 1$, hence $g \in \widehat{P}$ ;
this implies $\sigma = 1$ as we showed before. But using again $\delta : X/W \to \mathbbm{P}^1 \setminus
\{ 0, \infty \}$ we have $\sigma \in G(\Q(\mu_{\infty}))$. Since $K \subset \Q(\mu_{\infty})$
this concludes the proof.

\end{proof}

\subsection{A remark on complex conjugation}

A consequence of theorem \ref{theorat} is that the complex conjugation,
being a continuous automorphism of $\C$, induces an automorphism $\sigma$ of
$B = \pi_1(X/W,z)$, for $z$ a \emph{real} point in $X/W$.
Such an automorphism setwise stabilizes $P$ and
is well-defined up to inner automorphism. In case
$W$ is a Coxeter group, one can take for base point an arbitrary point inside the
Weyl chamber. Then $B$ has a classical Tits-Brieskorn presentation,
with generators $s_1,\dots,s_n$ and Coxeter-like relations. Geometrically,
it is clear that $\sigma$ maps $s_i \mapsto s_i^{-1}$. It coincides with the well-known `mirror image'
automorphism in the theory
of Artin-Tits groups.

In the complex setting however, the question of finding Artin-like
presentations for $B$ has been the subject of intense studies in the past decade,
starting from the partly conjectural presentations of \cite{BMR}.
The complex conjugation is easier to see on some presentations
than others (recall that its image in $\Aut(W)$ has been studied
in \cite{MARINMICHEL}). We do an example here.

For instance, let us consider the group $W$ of rank 2 and type $G_{12}$.
Then $B$ has presentation $<s,t,u \ | \ stus = tust = ustu >$,
and one can check on this presentation
that it admits an automorphism of order 2 given by $s \mapsto t^{-1},
t \mapsto s^{-1}, u \mapsto u^{-1}$. The identification of this automorphism with the
complex conjugation asks for getting back to the source of this
presentation, which lies in \cite{BANNAI}, and to redo
it in an algebraic and Galois-equivariant way. Redoing the computations
from the $\Gal(K'' | \Q)$-invariant models provided
in \cite{MARINMICHEL}, one finds $\C[x_1,x_2]^W = \C[\alpha,\beta]$ with
$\alpha = (x_1^2 - x_2^2)(x_1^4 + 12 x_1^2 x_2^2 + 4 x_2^4)$
and $\beta = (x_1^2-4x_1x_2-2x_2^2)(x_1^2 + 4 x_1 x_2 - 2 x_2^2)(3 x_1^4
+ 4 x_1^2 x_2^2 + 12 x_2^4)$, and the `discriminantal equation'
defining $X/W$ inside $\Spec \C[\alpha,\beta]$
is $0 = \beta^3 - 27 \alpha^4$. Up to a rational
change in coordinates one gets that the identification
of $X/W$ with $\{ (z_1,z_2) \ | \ z_2^3 \neq z_1^4 \}$
given by \cite{BANNAI} is indeed over $\Q$ and Galois-equivariant.
One can then consider the rational morphism $f : X/W \to \mathbbm{A}^1$
which maps $(z_1,z_2) \mapsto z_1$, which has for typical
fiber $\mathbbm{A}^1 \setminus \mu_3$. Then Bannai
shows that $B$ is generated by the image of the $\pi_1$ of this fiber,
and the chosen generators $s,t,u$ are simple loops around
the points of $\mu_3(\C)$. Taking the base point on
the real line of absolute value $> 1$, we gets that the automorphism
described above is indeed the image of the complex conjugation.
The issue is that, Bannai does in general \emph{not} choose the base point
on the real line. For $G_{12}$ this has little consequences (taking the
base point on the real line provides the formula above),
but for the other exceptional groups there is a need
in getting presentations which are `compatible' with the real
structure.

\begin{prob} Find presentations of the exceptional braid groups with `nice'
conjugation automorphism / Find explicit formulas for the conjugation
in terms of the known presentations of these groups.
\end{prob}

For the general series of the $G(e,e,n)$
a nice presentation for $B$ has been obtained
by R. Corran and M. Picantin in \cite{CORRANPICANTIN},
with generators $t_i , i \in \Z/e\Z$, $s_3,s_4,\dots,s_n$
and relations $t_{i+1} t_i = t_{j+1}t_j$, $s_3 t_i s_3 = t_i s_3 t_i$
and ordinary braid relations between the $s_3,\dots,s_n$.

It is deduced by algebraic methods from one given in \cite{BMR}. We gave in \cite{HOMOLOGY}
a topological description of this presentation, with base point in the
real part of $X/W$. Since
all the morphisms used there are algebraic over $\Q$,
it is easily checked that this presentation behaves
nicely with respect to the complex conjugation,
and we get the explicit formulas $s_k \mapsto s_k^{-1}$,
$t_i \mapsto t_{-i}^{-1}$.

The following question was asked to me by F. Digne. It is well-known that
non-isomorphic $W$ may provide the `same' braid group $B$. This happens
for instance when the spaces $X/W$ are analytically isomorphic, in which case both complex
conjugations provide the same element in $\Out(B)$. However, there are some other
coincidences where geometry does not seem to help decide.

\begin{quest} Is it possible to get two reflection groups $W_1,W_2$ with
isomorphic braid group $B$, such that the complex conjugation induces
two distinct elements in $\Out(B)$ ?
\end{quest}

This question should be viewed as a test about the comprehension of
the coincidence phenomena (see question \ref{questcoinc} below).

A good test for a positive answer to this question seemed to be
the exceptional group called $G_{13}$ in the Shephard-Todd
classification. It has rank 2, and its braid group $B$ is isomorphic
to the Artin group of type $I_2(6)$, in a seemingly
non-geometric way. Moreover, $\Out(B)$ is rather large
(it is precisely isomorphic to $(\Z \rtimes (\Z/2)) \times (\Z/2)$, see \cite{CRISPPARIS}).
For $G_{13}$, $Y_1 = X/W = \{ (x,y) \ | \ y(x^3-y^2) \neq 0 \}$,
whereas for the dihedral group $I_2(6)$, 
$Y_2 = X/W = \{ (x,y) \ | \ x^6 \neq y^2 \}$. By Zariski-Van Kampen
method Bannai found the presentations
$< g_1,g_2,g_3 \ | \ g_1 g_2 g_3 g_1 = g_3 g_1 g_2 g_3, g_3 g_1 g_2 g_3 g_2 = g_2 g_3 g_1 g_2 g_3 >$
coming from $Y_1$, whereas the presentation arising from $Y_2$
is the usual $<a , b \ | \ ababab = bababa >$. An explicit automorphism is
given by $a \mapsto g_3 g_1 g_2 g_3$, 
$b \mapsto g_3^{-1}$.

Although $Y_1$ and $Y_2$ are homotopically equivalent (as they are $K(\pi,1)$ with
the same $\pi_1$) they do not seem to be geometrically related.
Surprisingly enough, an explicit computation shows that the complex conjugation
provides the same element in $\Out(\hat{B})$. Indeed, from a suitable choice
of a base point one gets $g_1 \mapsto g_1^{-1}$, $g_2 \mapsto g_1 g_2^{-1} g_1^{-1}$
and $g_3 \mapsto g_1 g_2 g_3 g_2^{-1} g_1^{-1}$, which, under the
above isomorphism, translates into
$a \mapsto (bab)a^{-1} (bab)^{-1}$, $b \mapsto (ba)b^{-1} (ba)^{-1} = (bab)b^{-1} (bab)^{-1}$,
which has clearly same image in $\Out(B)$ as the complex conjugation for $I_2(6)$.

\subsection{Group-theoretic properties of complex braid groups}

We consider now the possible goodness of these groups,
in the sense of Serre (\cite{SERRE} \S 2.6). We recall from there
that a group $G$ is called good if, for every
finite $G$-module $M$ and positive integer $q$, the natural maps
$H^q(\widehat{G},M) \to H^q(G,M)$ are isomorphisms. Goodness
is inherited by finite-index subgroups, and a crucial property
is that, if $G$ is an extension of $Q$ by $H \triangleleft G$,
then $G$ is good as soon as $H,Q$ are good, when $H$ is
finitely generated and its cohomology groups with coefficients
in finite $H$-modules are finite. This last assumption is
clearly verified by groups of type $FP_{m}$ for some $m$, that is a group $H$
such that the trivial $\Z H$-module $\Z$ admits a resolution
of finite length by free modules of finite rank. This obviously holds
for the fundamental group of an affine complex algebraic variety
of dimension $m$, as it is homotopically equivalent to
a finite CW-complex of (real) dimension $m$.

A well-known consequence is that iterated extension of free groups
are good. Recall that the profinite completion functor is right exact,
and that every short exact sequence  $1 \to H \to G \to K \to 1$
yields a short exact sequence $1 \to \widehat{H} \to \widehat{G} \to \widehat{K} \to 1$ when $H$
is finitely generated and $K$ is good.

\begin{prop} 
For all but possibly finitely many
irreducible reflection groups $W$, the groups $P$ and $B$ are good in the sense
of Serre. The possible exceptions are the ones
labelled
$G_{23}, G_{24},G_{27},G_{28},G_{29}$,
$G_{30},G_{31}$,
$G_{33},G_{34},G_{35},G_{36},G_{37}$ in the Shephard-Todd classification.
\end{prop}
\begin{proof}
Let $W$ be such an irreducible reflection group.
Since $P$ has finite index in $B$ the
properties for $P$ or $B$ to be good are equivalent.
We first prove that $P$ is good when $W = G(de,e,n)$.
In this case, we prove that $ P$ is actually an iterated extension
of free groups of finite rank.

When $d \neq 1$
this is the case because the corresponding hyperplane
complement $X_n = \{ \underline{z} \in \C^n | z_i \neq 0, z_j/z_i \not\in
\mu_{de}(\C) \mbox{ for } j \neq i \}$ is fiber-type in
the sense of Falk-Randell (or equivalently: supersolvable, see e.g. \cite{OT}).

We briefly recall the definition. The assertion that $X_n$
is fibertype means that there exists a fibration $X_n \to Y$ for some $Y \subset \C^{n-1}$
which is the restriction of a linear map, and such that $Y$ is an hyperplane
complement which is itself fibertype --- this inductive definition
makes sense because the dimension decreases by one, and because
in addition every hyperplane
arrangement is said to be fibertype in dimension 1. A classical consequence of
this property and of the long exact sequence of a fibration is
that (1) $X_n$ is a $K(\pi,1)$ and (2) there is a short exact
sequence $1 \to \pi_1(F) \to \pi_1(X_n) \to \pi_1(Y) \to 1$,
with the fiber $F$ being the complement in $\C$ of a finite number
of points. From this it easily follows that $\pi_1(X_n)$
is an iterated extension
of free groups of finite rank. A classical fact is
that central hyperplane complements of rank 2 (i.e.
complement of a finite collection of linear hyperplanes
in $\C^2$) are always fibertype (a convenient projection
$\C^2 \to \C \setminus \{ 0 \}$ is given by any of the linear forms
defining an hyperplane of the arrangement).  

In our case, when $d \neq 1$, the fibration is given by
$X_n \to X_{n-1}$, $\underline{z} \mapsto
(z_1,\dots,z_{n-1})$. In case $d=1$, we cannot
use the fibertype machinery. However, the hyperplane
complement $X_n = \{ \underline{z} \in \C^n | , z_j \not\in
\mu_{e}(\C)z_i \mbox{ for } j \neq i \}$ admits a fibration
$\underline{z} \mapsto (z_1^e- z_n^e,\dots,z_{n-1}^e- z_n^e)$ over the
space $Y = \{ \underline{z} \in \C^n \ |  z_i \neq 0 , z_i \neq z_j
 \mbox{ for } j \neq i \}$, which is fiber-type (consider the
natural map $\underline{z} \mapsto
(z_1,\dots,z_{n-1})$ on $Y$). 
It follows that $\pi_1(Y)$ is good
and that $Y$ is a $K(\pi,1)$, hence from the homotopy sequence of
the fibration it follows that $P$ is an extension of $\pi_1(Y)$
by the $\pi_1$ of the fiber, which is clearly a smooth algebraic curve.
The group $\pi_1(Y)$ is finitely generated
of type $FP_m$ for some $m$ (e.g. because $Y$ is a $K(\pi,1)$ with finite
cellular decomposition), hence its cohomology groups with values
in any finite module are finite. Since $\pi_1$'s of curves are good,
it follows (see \cite{SERRE}) that $P$ is good (alternatively, an explicit
decomposition as an iterated extension of free groups can be found in
\cite{BMR} proposition 3.37).

We now recall that, among the exceptional reflection groups, half of them, namely
the ones labelled $G_4$ to $G_{22}$
by Shephard and Todd \cite{ST}, have rank 2. As noticed before,
this implies that
the associated hyperplane complement is necessarily fibertype,
hence that $P$ is good.

Among the remaining groups, which are labelled $G_{24}$ to $G_{37}$
in the Shephard-Todd classification, three of them have the same braid
group than one in the infinite series. Specifically we get the same braid
group $B$ for the pairs $(G_{25},G(1,1,4))$,
$(G_{32},G(1,1,5))$,
$(G_{26},G(2,2,3))$ (see e.g. the tables of \cite{BMR}).
By the above arguments, this proves that the corresponding braid
groups $B$ are good.
There remains the list of 12 possible exceptions mentionned in the
statement.
\end{proof}

In general, the following conjecture however remains open.
\begin{conj} \label{conjgood} $B$ is good in the sense of Serre.
\end{conj}

We recall from e.g. \cite{OT} that the Shephard groups are the symmetry groups
of regular complex polytopes. Their braid group is always isomorphic
to the braid group of another complex reflection group which belongs to the infinite series,
and which is also a Coxeter group.
Among the exceptional groups, and in rank at least $3$, these are the groups labelled $G_{25}$,
$G_{26}$ and $G_{32}$ in the Shephard-Todd classification. 
One can prove the following.

\begin{prop}
Let $W$ be an irreducible finite reflection group. The groups
$B$ and $P$ are residually finite when $W = G(de,e,n)$ for $d \neq 1$,
when $W$ is a Coxeter or a Shephard group, and when $W$ has rank 2.
\end{prop}
\begin{proof}
Again, the statements for $P$ and $B$ are equivalent.
In case $W = G(de,e,n)$ for $d \neq 1$ or $W$ has rank 2, then
$P$ is residually finite because it is residually torsion-free
nilpotent ; indeed, this is the case for the $\pi_1$ of an arbitrary fibertype hyperplane complement
(see \cite{FALKDEUX}).
In case $W$ is a Coxeter group (and thus also when $W$ is a Shephard group) this is a consequence of the linearity
of $B$, proved by F. Digne and A. Cohen - D. Wales (see \cite{DIGNE,COHENWALES}),
since finitely generated linear groups are residually finite.
\end{proof}

In the remaining cases, the groups $P$ are conjectured to be
residually torsion-free nilpotent,
and the groups $B$
to be linear -- these both statements being potentially connected,
as shown in \cite{KRAMCRG} and \cite{RESFREE}, and both of them implying the residual finiteness of $B$.
The following conjecture is thus highly plausible.

\begin{conj} $B$ is residually finite.
\end{conj}

Another group-theoretic question concerns the centre of these group and of their profinite
completion. It was conjectured in \cite{BMR} that, when $W$ is irreducible, $Z(P)$ and $Z(B)$ were infinite cyclic,
generated by natural elements denoted $\pi$ and $\beta$, and that
they are related to each other through an exact sequence $1 \to Z(P) \to Z(B) \to Z(W) \to 1$. This conjecture
has been recently proved (see \cite{DIGNEMARINMICHEL}). In addition, we proved in \cite{DIGNEMARINMICHEL}
that, whenever $U$ is a finite index subgroup of $B$, one has $Z(U) \subset Z(B)$.
This raises the following natural question :

\begin{quest} Is $Z(\hat{B})$ (respectively $Z(\hat{P})$) freely generated by  $\beta$ (resp. $\pi$) ? Do we
have a short exact sequence $1 \to Z(\hat{P}) \to Z(\hat{B}) \to Z(W) \to 1$ ?
\end{quest}

Finally, from the point of view of capturing the image of $G(\Q)$ inside $\Out \hat{B}$, one would like to know the following.

\begin{quest} \label{questcoinc} Which reflection groups $W$ actually
provide distinct (non-isomorphic) braid groups $B$, and more importantly distinct
profinite groups $\widehat{B}$ ?
\end{quest}

A partial answer has been
obtained in \cite{HOMOLOGY}, for the case of reflection groups
having one conjugacy class of reflections : the map $W \mapsto B$
is injective on irreducible 2-reflection groups with one class of reflections.
Actually the methods used there are either cohomological
or depend only on $\widehat{B}$, so it also proves the
injectivity of $W \mapsto \hat{B}$ under the goodness assumption (conjecture \ref{conjgood} above).

\section{Grothendieck-Teichm\"uller actions}

A few years ago, P. Lochak asked the following question, in the setting
of real reflection groups. It now seems natural to extend this question to complex reflection groups.
\begin{quest} For $W$ a complex reflection group,
is there an action of $\GTH$ on $\widehat{B}$ which extends
an action of $G(\Q)$ ? Or of $\GGam$, $\GGam'$,\dots ?
\end{quest}

Here $\GTH$ denotes the usual Grothendieck-Teichm\"uller, whose definition we recall below, and $\GGam$, $\GGam'$ denote
subgroups of $\GTH$ (whose equality with $\GTH$ is an open question) which act on the full
tower of mapping class groups in all genus. These groups have been defined in \cite{NAKASCH}.

In this section we shall explore this question (without answering it in any way)
in the general setting of complex reflection groups,
and we shall try to distinguish a few maybe more accessible subquestions.

\subsection{The Grothendieck-Teichm\"uller group}

We denote $Br_n$ the usual braid group on $n$ strands, with
generators $s_1,\dots,s_{n-1}$, and denote $\widehat{GT}$
the (profinite version of) the Grothendieck-Teichm\"uller group,
introduced by Drinfeld in \cite{DRINFELD}. Recall that the elements
of $\widehat{GT}$ are couples $(\la, f) \in \widehat{\Z}^{\times}
\times \widehat{F}_2$, where $F_2$ denotes the free group
of rank 2, and that there exists an embedding $G(\Q) \into \widehat{GT}$,
$\sigma \mapsto (\chi(\sigma),f_{\sigma})$ with $\chi : G(\Q) \to \widehat{\Z}^{\times}$
the cyclotomic character.

There is a natural (Drinfeld) action of $\widehat{GT}$ on $\widehat{Br_{n}}$
defined in \cite{DRINFELD}, that associates
to $(\la,f) \in \widehat{GT}$ the automorphism $s_1 \mapsto s_1^{\la}$,
$s_i \mapsto f(s_i^2,y_i) s_i^{\la} f(y_i,s_i^2)$.

\subsection{A remark on the derived subgroup}

One of the properties of the couples $(\la,f) \in \widehat{GT}$ is
that $f \in \widehat{F}_2$ actually belongs to the derived subgroup
$(\widehat{F}_2,\widehat{F}_2)$. Here, we let $(G,G)$
denote the (algebraic) derived subgroup of the group $G$, that
is the subgroup algebraically generated by the commutators
$(a,b) = a b a^{-1} b^{-1}$ for $a,b \in G$. Note that, for
a finitely generated group $G$, there are a priori two natural
notions for the `derived subgroup' of $\widehat{G}$. Letting
$p : G \onto G^{ab}$ denote the abelianization morphism,
and $\widehat{p} : \widehat{G} \onto \widehat{G^{ab}}$ the induced
morphism, $\Ker \widehat{p}$ equals the closure $\overline{(G,G)}$
of (the image of) $(G,G)$ in $\widehat{G}$. Of course $\widehat{G^{ab}}$
is abelian, hence $(\widehat{G},\widehat{G}) \subset \Ker \widehat{p} = 
\overline{(G,G)}$, and clearly $(G,G) \subset (\widehat{G},\widehat{G})$.
When $G$ is finitely generated, by the recent result of \cite{NIKOLO} theorem
1.4, $(\widehat{G},\widehat{G})$ is closed hence $(\widehat{G},\widehat{G})
= \overline{(G,G)} = \Ker \widehat{p}$.
Moreover, 
$G^{ab}$ being a finitely generated abelian group is good,
hence the
sequence $1 \to \widehat{(G,G)} \to \widehat{G} \to \widehat{G^{ab}}\to 1$
is exact. In our case $G = F_2$ we thus get
$\widehat{(F_2,F_2)} \simeq \overline{(F_2,F_2)} = 
(\widehat{F}_2,\widehat{F}_2)$

\subsection{Artin group of type $B_n$ and the groups $G(de,e,n)$, $d > 1$}

For the complex reflection groups $G(de,e,n)$ with $d \neq 1$,
the corresponding braid group is a finite index subgroup of the Artin group
$Art(B_n)$ of type $B_n$ (see \cite{BMR} proposition 3.8). Recall that this group is generated by elements $t, \ss_2,
\ss_3, \dots, \ss_{n}$ with usual braid relations between the $s_i$'s, and $ts_2ts_2 = s_2 t s_2 t$,
$t s_i = s_i t$ for $i \geq 3$. This group can be embedded
in $Br_{n+1}$ through $t \mapsto s_1^2$, $\ss_i \mapsto s_i$.
We denote the image subgroup by $Br_{n+1}^1$. It has finite index $n+1$
in $Br_{n+1}$, being the inverse image in $Br_{n+1}$ of
the permutations fixing $1$. This implies that $Br_{n+1}^1$
is good, and also that the closure of $Br_{n+1}^1$ in $\widehat{Br_{n+1}}$
can be identified with $\widehat{Br_{n+1}^1}$. From
the formulas defining the $\widehat{GT}$ action on $Br_{n+1}$
it is then clear that $\widehat{Br_{n+1}^1}$ is stabilized.
This action of $\widehat{GT}$ is compatible with the morphism
$G(\Q) \into \widehat{GT}$ and with the action of $G(\Q)$ on
$\widehat{\pi}_1(X/W)$, for $W$ the Coxeter group of type
$B_n$, with respect to a natural tangential base-point (see \cite{MATSUARTIN}).
Related results on $Art(B_n) \simeq Br_{n+1}^1$ can be found in
\cite{LS}.

When $W = G(de,e,n)$ with $d > 1$, the space $X$ depends
only on $de$ and $n$, while $X/W$ depends only on $e$ and $n$.
The corresponding braid group $B$ is then the kernel
of $Art(B_n) \onto \Z/e$ defined by $t \mapsto 1$, $\ss_i \mapsto 0$,
and as before $\widehat{B}$ can be identified with
the kernel of the induced morphism $\widehat{Art(B_n)} \onto \Z/e$.
Clearly the action of $\widehat{GT}$ on $\widehat{Art(B_n)}$
stabilizes this kernel, hence an action of $\widehat{GT}$
on $\widehat{B}$.

\subsection{Artin group of type $D_n$ and the groups $G(e,e,n)$}

\subsubsection{On the Artin group of type $D_n$}

The Artin group of type $D_n$ is generated by elements $\ss_1, \ss'_1,\ss_2,
\dots,\ss_{n-1}$ with both $\ss_1,\ss_2,\dots,\ss_{n-1}$
and $\ss'_1,\ss_2,\dots,\ss_{n-1}$ satisfying the relations
of $Br_n$, plus $\ss_1 \ss'_1 = \ss'_1 \ss_1$. When $2 \leq r \leq n$
we identify $Art(D_r)$ with the subgroup of $Art(D_n)$
generated by $\ss_1,\ss'_1,\ss_2,\dots,\ss_{r-1}$.

We let $w_r = \ss_1 \ss'_1 \ss_2 \dots \ss_r$ when $r \geq 2$.
The Garside element $\Delta_r \in Art(D_{r+1})$ is $w_r^{r-1}$. It
is central in $Art(D_{r+1})$ when $r$ is even, and in all cases $\Delta_r^2$ is
central in $Art(D_r)$. When $r$ is odd, conjugation by $\Delta_r$
in $Art(D_{r+1})$ induces the `diagram' automorphism $\phi$ which exchanges
$\ss_1$ with $\ss'_1$ and fixes the $\ss_i$ for $i > 1$. We let $\eta_r
 = \ss_{r-1} \dots \ss_2 \ss_1 \ss'_1 \ss_2 \dots \ss_{r-1}$.

The proof of the following technical lemma is easy by induction,
and left to the reader. The items (1) and (2) only make easier
the proof of (3), which is the item we need for the sequel.
\begin{lemma} \label{lemcalcD}
\begin{enumerate}
\item $\forall r \geq 3 \ \ w_{r+1} \ss_{r-1} = \ss_r w_{r+1} $
\item $\forall r > i > 2 \ \ w_{r+1} \ss_{r-1} \dots \ss_i = \ss_r \dots
\ss_{i+1} w_{r+1}$
\item $\forall r \geq 2 \ \ \Delta_r \eta_{r+1} = \Delta_{r+1}$.
\end{enumerate}
\end{lemma}

The following lemma is then an easy consequence of item (3) above and
of the properties of $\Delta_r$ recalled earlier.

\begin{lemma}
\begin{enumerate} \label{lemcommD}
\item If $x \in Art(D_r)$ centralizes $\ss_1,\ss'_1$, then so
does $(\eta_r^m,x)$ for all $m \in \Z$, whenever $r \geq 3$.
\item If $f \in (\widehat{F}_2,\widehat{F}_2)$ and $r \geq 3$,
then $f(\eta_r,\ss_r^2)$ centralizes $\ss_1,\ss'_1$ in $\widehat{Art(D_r)}$.
\end{enumerate}
\end{lemma}

\begin{proof}
(1) By lemma \ref{lemcalcD} (3) we have $\eta_r = \Delta_{r-1}^{-1} \Delta_r$,
hence $\eta_r y \eta_r^{-1} = \phi(y)$ for all $y$.  Note that $\phi^2 = \mathrm{Id}$ and
thus $\phi(y) = \eta_r^{-1} y \eta_r$.

It follows that,
for $y \in \{ \ss_1, \ss'_1 \}$ and $m \in \Z$,
$$
\begin{array}{lclcl}
(\eta_r^m,x) y &=& \eta_r^m x \eta_r^{-m} x^{-1} y &=& \eta_r^m x \eta_r^{-m} y x^{-1} \\
&=& \eta_r^m x \eta_r^{-m} y \eta_r^m \eta_r^{-m} x^{-1} &=& \eta_r^m x \phi^m (y) \eta_r^{-m} x^{-1} \\
&=& \eta_r^{m}\phi^m(y) x \eta_r^{-m}x^{-1} &=& \eta_r^{m} \phi^m(y) \eta_r^{-m}(\eta_r^{m},x) \\
&=& \phi^{2m}(y) (\eta_r^{m},x) &=& y (\eta_r^{m},x) 
\end{array}
$$
(2) We denote $C$ the centralizer in $Art(D_r)$
of $< \ss_1, \ss'_1 >$. We let $H_k$ denote the subgroup of $(F_2,F_2)$
generated by the commutators of length at most $k$ in
powers of $\ss_r^2$ and $\eta_r$, and
$\varphi : (F_2,F_2) \to Art(D_r)$ defined
by $f \mapsto f(\eta_r,\ss_r^2)$. By induction on $k \geq 2$,
we show that $\varphi(H_k) \subset C$. The case $k = 2$
is a consequence of (1), as $(\eta_r^u, \ss_r^v) \in C$ for all $u,v \in \Z$ by (1). 
Let now $c$ be (the image under $\varphi$ of) a commutator of length $k+1$. Either $c = (u,v)$
with $u,v \in H_k \subset C$, and then $c \in C$, or
$c = (\ss_r^{2m}, v)$ for some $m \in \Z \setminus \{ 0 \}$ with $v \in H_k \subset C$ and then $c \in C$
because $\ss_r^2 \in C$, or $c = (\eta_r^m, v)$ for some $m \in \Z \setminus \{ 0 \}$ with $v \in H_k \subset C$,
and then $c \in C$ by (1). This proves that
$\varphi((F_2,F_2)) \subset C$, that is $\varphi(f) \ss_1 = \ss_1 \varphi(f)$,
$\varphi(f) \ss'_1 = \ss'_1 \varphi(f)$ for all $f \in (F_2,F_2)$,
hence $f(\eta_r,\ss_r^2)$ centralizes $\ss_1, \ss'_1$ in $\widehat{Art(D_r)}$
for all $f$ in the closure of $(F_2,F_2)$ in $\widehat{F}_2$.
\end{proof}

\subsubsection{Lifting the Galois action for $D_n$ to $\widehat{GT}$.}

In \cite{MATSUARTIN}, Matsumoto showed that, for $W = D_n = G(2,2,n)$,
the action of $G(\Q)$ on $X$ with respect to a natural tangential basepoint
can be expressed with
the formulas $\ss_1 \mapsto \ss_1^{\la}, \ss'_1 \mapsto (\ss'_1)^{\la}$,
$\ss_i \mapsto f(\ss_i^2,\eta_i) \ss_i^{\la} f(\eta_i,\ss_i^2)$
for $i \geq 2$, with $(\la,f)$ the image in  $\widehat{GT}$ of the
element of $G(\Q)$. The question is thus to check whether the morphism $F$
from the free group on $\ss_1,\ss'_1,\ss_2, \dots, \ss_{n-1}$
to $\widehat{Art(D_n)}$ defined by these equations, for an arbitrary
$(\la,f) \in \widehat{GT}$, factors through the natural morphisms
$F \to Art(D_n)$ and $Art(D_n) \to \widehat{Art(D_n)}$. Using the previous results we find that the \emph{commutation
relations} are satisfied.
\begin{prop}
$F(\ss_i) F(\ss_j) = F(\ss_j) F(\ss_i)$ when $|j-i| \geq 2$
and $F(\ss_1') F(\ss_j) = F(\ss_j) F(\ss_1')$ when $j \geq 3$
\end{prop}
\begin{proof}
When $i,j \geq 2$ for the first equation it follows from the fact
that $\eta_i$ commutes with $\ss_j$ in these cases. We can
thus assume $i = 1$ for the first equation. It is then
a consequence of the fact that $f \in (\widehat{F_2},\widehat{F_2})$
by lemma \ref{lemcommD} (2).
\end{proof}

The braid relation $F(\ss_1) F(\ss_2) F(\ss_1) = F(\ss_2) F(\ss_1) F(\ss_2)$
does not, at least readily, follow from the defining equation
of $\widehat{GT}$. Indeed, the restriction to $<\ss_1,\ss_2> \simeq Art(A_2)$
from $Art(D_n)$ of the action of $G(\Q)$ is $\ss_1 \mapsto \ss_1^{\la}$,
$\ss_2 \mapsto f(\ss_2^2,\ss_1 \ss'_1) \ss_2^{\la} f(\ss_1 \ss'_1,\ss_2^2)$,
whereas the natural (Drinfeld) action of $\widehat{GT}$ on
$Br_3$ is $\ss_1 \mapsto \ss_1^{\la}$, $\ss_2 \mapsto f(\ss_2^2 ,\ss_1^2)
\ss_2^{\la} f(\ss_1^2,\ss_2^2)$. However, identifying
$Art(D_3) = <\ss_1,\ss_2,\ss_1'>$ with $Art(A_3) = Br_4$ we notice
that the two
actions almost coincide when $(\la,f)$ belong to the
subgroup $\GGam' \subset \widehat{GT}$ defined in \cite{NAKASCH}.
Indeed, recall from \cite{NAKASCH} lemma 4.2 that, when $(\la,f) \in \GGam'$,
its natural action on the image of $\widehat{Br_4}$ modulo center
\footnote{This embedding is defined in \S 3 of \cite{NAKASCH} p. 518. The ambiguity here is only apparent, as
the short exact sequence $1 \to Z(Br_4) \to Br_4 \to Br_4/Z(Br_4) \to 1$
survives the profinite completion : $Br_n/Z(Br_n)$ is good,
because of the short exact sequences $1 \to P_n/Z(P_n) \to Br_n/Z(Br_n)
\to \mathfrak{S}_n \to 1$ and $1 \to F_{n-1} \to P_n/Z(P_n) \to P_{n-1}/Z(P_{n-1}) \to 1$}
inside the profinite mapping class group $\widehat{\Gamma}_0^{[5]}$
composed with $\mathrm{Ad} f(x_{45},x_{34})$
satisfies $\ss_2 \mapsto f(\ss_2^2,\ss_1\ss_3) \ss_2^{\la} f(\ss_1 \ss_3,\ss_2^2)$.
Here $x_{34} = \ss_3^2$ and $x_{45} = (\ss_1 \ss_2)^3 = \ss_2 \ss_1^2 \ss_2.
\ss_1^2$.  By the formulas (4.1) of \cite{NAKASCH} we also have
$\ss_1 \mapsto \ss_1^{\la}$.
Under the isomorphism $Art(D_3) \simeq Art(A_3) \simeq Br_4$ this
means $F(\ss_1) F(\ss_2) F(\ss_1) = F(\ss_2) F(\ss_1) F(\ss_2) \omega_4^{\mu}$
for some $\mu \in \widehat{\Z}$. Taking the image under
$\widehat{Br_4} \onto \widehat{\Z}$ one gets $3 \la = 3 \la + 12 \mu$
in $\widehat{\Z}$ hence $\mu = 0$. In particular, the formulas of Matsumoto
for the Galois action on $\widehat{Art(D_3)}$ actually define an action of $\GGam'$.

Since the condition denoted (III)' in \cite{NAKASCH} is the starting
point for extending the action of $\widehat{GT}$ to mapping class
groups of higher genus, extending the action of $G(\Q)$
on $Art(D_n)$ to $\widehat{GT}$ seems related to the open question
'$\widehat{GT} = \GGam = \GGam'$ ?' of \cite{NAKASCH,HLS}. 
Explanations for this connection can be found in
the relationship between the mapping class group
$\Gamma_{1,3}$ and $Art(D_4)$ (using the Dehn twists pictured in figure \ref{figD4M3})
as well as in the known embeddings of $Art(D_n)$
into mapping class groups of higher genus. 
\begin{figure}
\begin{center}
\resizebox{6cm}{!}{\includegraphics{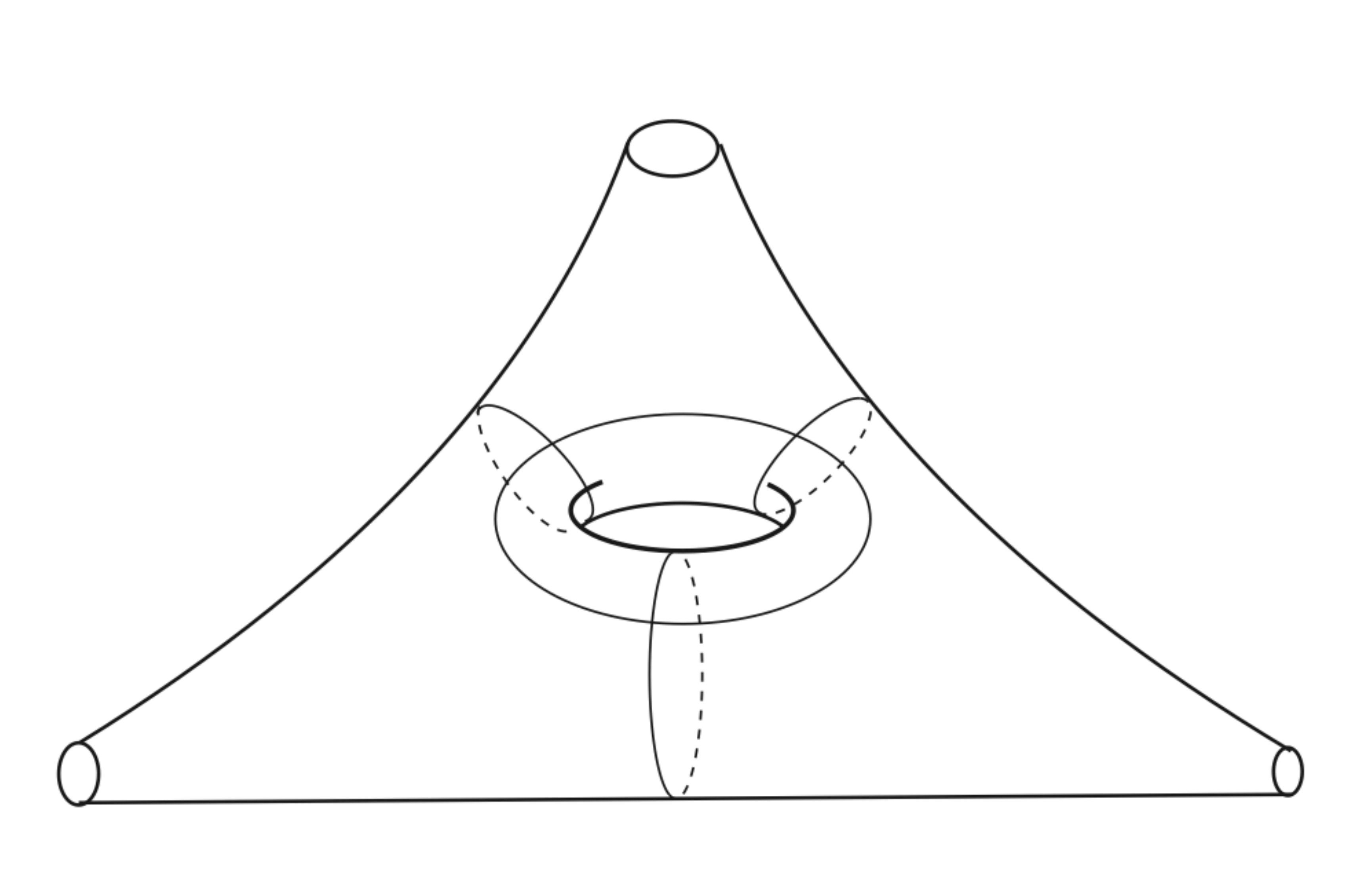}}
\end{center}
\caption{$D_4$ and $\Gamma_{1,3}$}
\label{figD4M3}
\end{figure}

Other connections
between mapping class groups and Artin groups of type $ADE$
can be found in \cite{MATSUARTIN} and \cite{MATSUANN}. Since nice presentations
for the braid groups of type $G(e,e,n)$ are now
at disposal, which generalize nicely the case of type $D_n$,
a natural question is the following.

\begin{subquest} What is the analogous picture for the more general series of groups $G(e,e,n)$ ?
Is there a related connection with mapping class groups ?
\end{subquest}

We hope to come back to this question in a forthcoming work.
A first step towards answering it would be to tackle
seriously the case of the groups $W = G(e,e,2)$. In this case
$B$ is generated by 2 generators $s,t$ with $stst\dots = tsts\dots$
($m$ factors on both sides), and the pure braid group $P$
has for derived subgroup a free group on $m$ generators. A group
capturing the Galois action will typically map $s \mapsto s^{\la}$,
$t \mapsto g^{-1} t^{\la} g$ for some $g \in \widehat{P}'$.
One can then define, following Drinfeld recipe for $\GTH$,
a group $\widehat{GD_m}$
made of elements $(\la,g)$ in $ \widehat{\Z}^{\times} \times \widehat{P}'$
such that there exists a morphism $\tilde{g}  \in \Hom(B,\widehat{B})$
satisfying
\begin{enumerate}
\item $\tilde{g}$ induces an automorphism of $\widehat{B}$
\item $\tilde{g}$ maps $s \mapsto s^{\la}$,
$t \mapsto g^{-1} t^{\la} g$ 
\item $\tilde{g}$ maps $\Delta \mapsto \Delta^{\la} g$ if $m$ odd,
$\Delta \mapsto \Delta^{\la}$ is $m$ even
\item $\tilde{g}$ maps $\Delta^2 \mapsto \Delta^{2\la}$.
\end{enumerate}
where $\Delta = stst\dots = tsts\dots$ ($m$ factors). 
In this group, composition is given by $(\la_1,g_1)\star (\la_2,g_2) = 
(\la_1 \la_2, g_1 \tilde{g}_1(g_2))$, and the conditions $(2)-(4)$
can be translated in algebraic conditions on the couple $(\la,g)$,
as in the analogous unipotent setting described in \cite{DIEDRAUX}.
The task is then to combine this setup with the one
of the usual braid group, in order to get a description for the
types $G(e,e,n)$, in such a way that it fits with the special
case of the $G(2,2,n)$. Note that a connection between these braid groups and
mapping class groups has not been found yet.

\subsubsection{The exceptional and Shephard groups.}

Of course, the remaining exceptional groups deserve equal interest, but
they are tricky to solve, as in essence they reflect very special geometric
situations. The case of Artin groups of type $E_n$, $n \in \{6,7,8 \}$ were described in
\cite{MATSUARTIN} : a description of the Galois action is given, as well as a connection
with mapping class groups of higher genus. It is shown in \cite{MATSUARTIN}
that the Galois action is compatible, in case of $W$ of type $E_7$, with the embedding
of $B$ inside the mapping class group in genus $3$. Even in that case, the question of
whether this action factorizes through $\GTH$ seems to be open -- which is not surprising
at all, as the Artin group of type $E_7$ admits a parabolic subgroup of type $D_4$.

However, there are some exceptional groups for which one \emph{knows} that
there exists an action of $\GTH$. These are groups for which the discriminantal
variety $X/W$ is isomorphic to the one for $W$ a symmetric group, or even $W$ a
Coxeter group of type $B_n$. 

These groups are the groups of symmetries of regular complex polytopes,
and are known as Shephard groups. They cover all the cases
where $X/W$ also arises from a Coxeter group.

More precisely, we have $W_1, W_2 < \GL(V)$
two different reflection groups, one of the two (say $W_1$) being a Coxeter group of type $A_n$ or $B_n$, such that $X_1/W_1 \simeq X_2/W_2$
(no such phenomenon appear for other Coxeter groups).
When $W_1$ is a symmetric group, of course the profinite completion of $B_2 \simeq B_1$ inherits an action
of $\GTH$. This happens exactly for $W_2$ of type $G_4,G_{25},G_{32}, G_8, G_{16}$ ; these groups
are the finite quotients of the usual braid group $Br_n$ by the relations $s_i^k = 1$ for some $k$, as shown by Coxeter in \cite{COXBANFF}.
(The case where $W_1$ is a Coxeter group of type $B_n$ corresponds to $W_2$ being of type $G_5,G_{10}, G_{18}, G_{26}$). We refer again to \cite{BMR}
for tables on these exceptional reflection groups and their braid groups, and to \cite{OT} \S 6.6 for the isomorphism $X_1/W_1 \simeq X_2/W_2$
in case of Shephard groups (see in particular cor. 6.126).

\begin{prob} Give a geometric interpretation of the $\GTH$-action on $X/W$, for $W = G_4,G_{25},G_{32}, G_8, G_{16}$.
Same question for $G_5,G_{10}, G_{18}, G_{26}$.
\end{prob} 

{\bf Example.} Let us consider the case of $W = G_4 < \GL_2(\C)$ generated by $\overline{s_1} = \begin{pmatrix} 1 & 0 \\ 0 & j \end{pmatrix}$
$\overline{s_2} = \frac{1}{\sqrt{-3}}\begin{pmatrix} -1 & j \\ 2 & j \end{pmatrix}$ with $j = \exp \frac{2 \ii \pi}{3}$ and $\sqrt{-3} = j-j^2$. Its braid group is generated by
lifts $s_1,s_2$ with relations $s_1 s_2 s_1 = s_2 s_1 s_2$, that is $B = Br_3$. We detail in this example a natural Belyi covering map
for $X/W$.
The complex conjugation acts
by $s_1 \mapsto s_1^{-1}, s_2 \mapsto s_1 s_2^{-1} s_1^{-1}$, and one finds by explicit computation
that $\C[z_1,z_2]^W = \C[g_1,g_2]$ with $g_1 = z_1^4 - z_1 z_2^3$,
$g_2 = z_1^6 + (5/2) z_1^3 z_2^3 - (1/8) z_2^6$, and an equation for $Y = X(G_4)/G_4$ inside $\Spec \C[g_1,g_2]$ is $g_1^3 \neq g_2^2$.
This has to be compared with the situation with $\mathfrak{S}_3 < \GL_2(\C)$, generated by $\hat{s}_1 = \begin{pmatrix} 1 & -1 \\ 0 & -1 \end{pmatrix}$
and $\hat{s}_2 = \begin{pmatrix} -1 & 0 \\ -1 & 1 \end{pmatrix}$ with $\C[z_1,z_2]^{\mathfrak{S}_3} = \C[f_1,f_2]$, $f_1 = z_1^2 - z_1 z_2 + z_2^2$,
$f_2 = z_1^3 -(3/2) z_1^2 z_2 - (3/2) z_1 z_2^2 + z_2^3$, $Y = X(\mathfrak{S}_3)/\mathfrak{S}_3$ having equation $f_1^3 \neq f_2^2$ in $\Spec \C[f_1,f_2]$ ; for the $\mathfrak{S}_3$-case
the space $X = X(\mathfrak{S}_3)$ is isomorphic to $C \times P$ under $(z_1,z_2) \mapsto (z_1,z_2/z_1)$ with $C = \A^1 \setminus \{ 0 \}$
and $P = \A^1 \setminus \{ 0 ,\infty \} = \P^1 \setminus \{ 0, 1 , \infty \}$. The \'etale covering $X(G_4) \to Y$ can then be pulled
back to an \'etale covering map $E  = X(\mathfrak{S}_3) \times_{Y} X(G_4) \to X(\mathfrak{S}_3)$ and then restricted to an \'etale covering map $E_0 \to P = \{ 1 \} \times P$
with fiber $G_4$.

(For the algebraically inclined reader, we notice that all computations can be done explicitely at the level of coordinate rings,
as all varieties used are affine. For instance, $E = \Spec S^{-1} A$ for $A = \C[u,v] \otimes \C[z_1,z_2]/g_i(u,v)=f_i(z_1,z_2)$
and $S = (g_1^3 - g_2^2)$, and $S_0 = \Spec A_0$ with $A_0 = (S^{-1} A) \otimes_{\C[z_1]} \C$ under $z_1 \mapsto 1$.
Given the explicit values of $g_1,g_2,f_1,f_2$ chosen above, this means $A_0 = \C[z^{-1},(z-1)^{-1},z,u,v]/I$, where
$I$ is generated by $u^4 - uv^3 - (1-z+z^2)$
and $u^6 + (5/2) u^3 v^3 - (1/8) v^6 - (1- (3/2) z - (3/2) z^2 + z^3)$.)

The situation can be summarized as the following sequence of morphisms between \'etale $G_4$-coverings
$$
\xymatrix{
E_0 \ar@{^{(}->}[r] \ar[dd]_{G_4} & E \ar[dd]_{G_4} \ar@{=}[r] & E \ar[dd]_{G_4}\ar[rr]^{\mathfrak{S}_3} & & X(G_4) \ar[dd]_{G_4} \\
 & & & \square \\
\{ 1 \} \times P\  \ar@{^{(}->}[r] & C \times P \ar@{=}[r] & X(\mathfrak{S}_3) \ar[rr]_{\mathfrak{S}_3} &  & Y
}
$$
In particular, the action of $\pi_1(P)$ on the fiber of $E_0 \to P$ can be computed on the fiber
of $X(G_4) \to Y$. For this, and since the morphism $Br_3 \to G_4$ is known, one only needs to compute
the image of $\pi_1(P) \to \pi_1(X/W) = Br_3$. We take $\frac{1}{2}$ for base-point in $P$, and denote
$x,y,z$ the usual generators of $\pi_1(P,\frac{1}{2})$ (see figure \ref{P1gens}). Figure \ref{figchambre} shows the image of $x$ and
$y$ in $\pi_1(X(\mathfrak{S}_3),(1,\frac{1}{2}))$ as well as homotopies (in green rays) with the standard loops around the
walls of
the Weyl chamber containing $(1,\frac{1}{2})$. The two points singled out in the pictures are the basepoint and its image under the
relevant reflection. We denoted $\hat{s}_3 = \hat{s_2} \hat{s}_1 \hat{s_2}= \hat{s_1} \hat{s}_2 \hat{s_1}$. The left-hand side of figure \ref{figchambre}
lives in $X(\mathfrak{S}_3) \cap \R^2 \oplus \ii \R (0,1)$, while the right-hand side lives
in $X(\mathfrak{S}_3) \cap \R^2 \oplus \ii \R (1,-1)$. It follows that, up to $Br_3$-conjugation, $x$ and $y$ are mapped to the squares of the
standard generators of $Br_3$. In particular, the usual generators $x,y,z = (xy)^{-1}$ act
(again, up to global conjugacy) by $\bar{s}_1^2, \bar{s}_2^2,(\bar{s}_1^2 \bar{s}_2^2)^{-1}$.
\begin{figure}
\begin{center}
\resizebox{6cm}{!}{\includegraphics{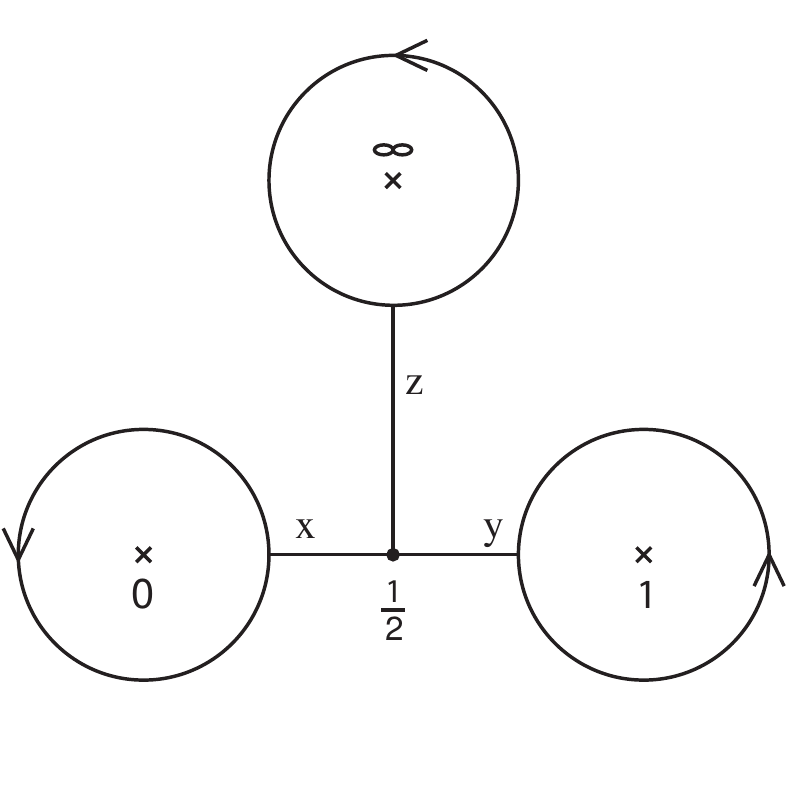}}
\end{center}
\caption{Generators of the $\pi_1$ of $\C \setminus \{ 0,1 \} = \mathbbm{P}^1 \setminus \{ 0,1 \infty \}$}
\label{P1gens}
\end{figure}

\begin{figure}
$$\resizebox{6cm}{!}{\includegraphics{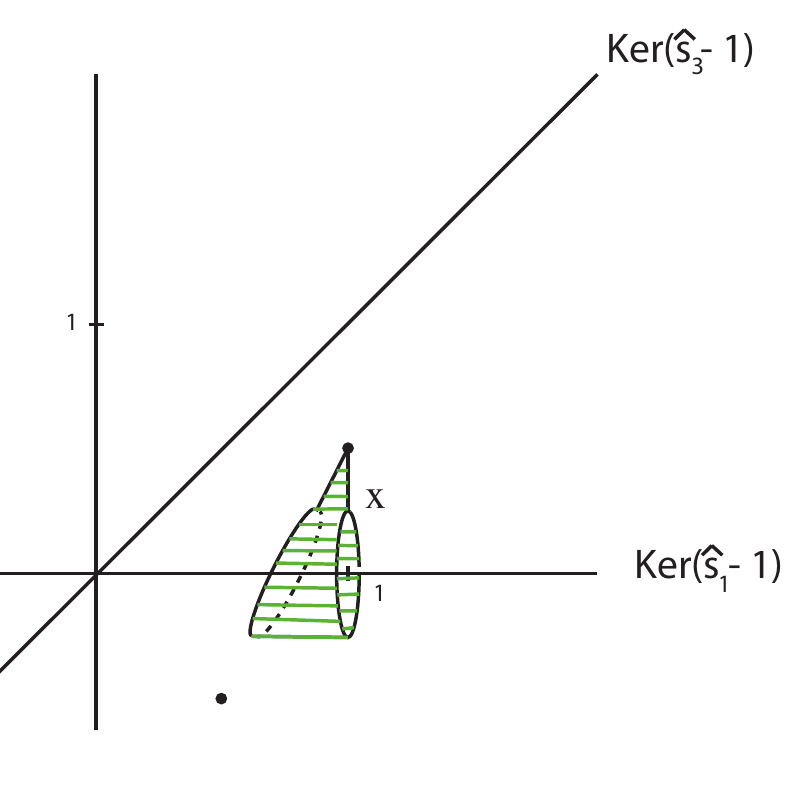}} \ \ \resizebox{6cm}{!}{\includegraphics{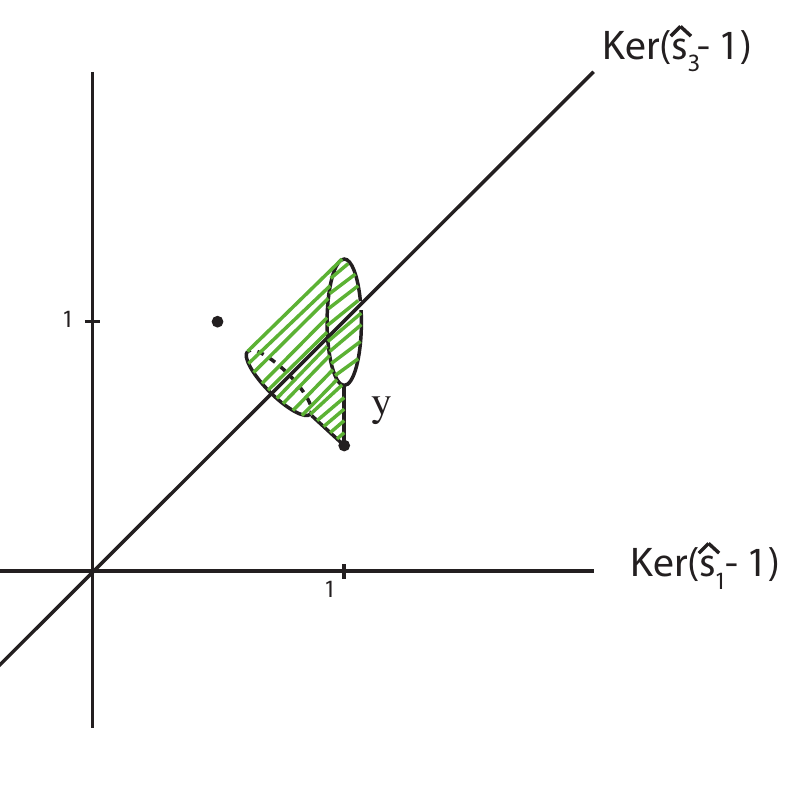}}
$$
\caption{Images of $x,y$ in $X(\mathfrak{S}_3)$ and homotopies with the standard loops.}
\label{figchambre}
\end{figure}

Extending this \'etale covering
of $P$ to a ramified covering map $\overline{E_0} \to \P^1$ we get the ramification at $0,1,\infty$ by considering the orders of these elements :
$\bar{s}_1^2$ and $\bar{s}_2^2$ have order $3$, whereas $(\bar{s}_1^2 \bar{s}_2^2)^{-1}$ has order $6$. This means that the fiber over
$0$ and $1$ has 8 elements with ramification index 3, whereas the fiber over infinity has 4 elements with ramification
index 6. Riemann-Hurwitz formula shows that $\overline{E_0}$ has genus $4$.



\end{document}